\documentclass{article}
\usepackage[numbers,sort&compress,merge]{natbib}
\bibpunct[, ]{(}{)}{,}{n}{,}{,}
\usepackage{mathtools, amsthm, amsfonts}
\usepackage{float}
\usepackage[font=scriptsize,labelformat=simple,labelsep=colon]{subcaption}
\usepackage[font={scriptsize},hypcap=true,labelsep=colon]{caption}
\theoremstyle{plain}
\newtheorem*{theorem*}{Theorem}

\DeclarePairedDelimiter\norm{\lVert}{\rVert}
\title{%
    \normalsize{\textbf{Properties of an Infinite Dimensional Banach Space over GF(2)}}\\[1em]
    \normalsize{Samuel Gomez, James Rose, Ryan Maguire}\\[0.2em]
    \normalsize{University of Massachusetts, Lowell}
}
\date{\vspace{-5em}}
\author{\vspace{-5em}}
\begin{document}
    \begingroup
    \let\center\flushleft
    \let\endcenter\endflushleft
    \maketitle
    \endgroup
    \vspace{4mm}
    \begin{abstract}
        A banach space $X$ is a normed vector space, which is complete with respect
        to the metric induced by the norm. Given a bounded linear operator $T$ acting
        on a banach space $X$, $T$ is said to attain its norm if there is a unit
        vector $\mathbf{x}\in{X}$, such that $\norm{T\mathbf{x}}=\norm{T}$. The existence
        of an infinite dimensional banach space $X$, in which each
        bounded linear operator acting on $X$ attains its norm, is still undetermined.
        This question was posed by M.I. Ostrovskii at St. John's University. In this
        paper we show that if an infinite dimensional banach space is considered
        over GF(2), then it is possible for every bounded
        linear operator to attain its norm.
    \end{abstract}
    \section*{Introduction}
    \normalsize
    The Galois Field of two elements, denoted GF(2), is the field
    containing 0 (zero) and 1 (one). The operations of
    addition and multiplication are defined as follows:
    \par
    \begin{minipage}[b]{0.4\textwidth}
        \centering
        \begin{table}[H]
            \centering
            \begin{tabular}{c|cc}
                +&0&1\\
                \hline
                0&0&1\\
                1&1&0
            \end{tabular}
            \caption{Addition in GF(2).}
            \label{tab:Addition_In_F2}
        \end{table}
    \end{minipage}
    \hfill
    \begin{minipage}[b]{0.4\textwidth}
        \centering
        \begin{table}[H]
            \centering
            \begin{tabular}{c|cc}
                $\cdot$&0&1\\
                \hline
                0&0&0\\
                1&0&1
            \end{tabular}
            \caption{Multiplication in GF(2).}
            \label{tab:Multiplication_In_F2}
        \end{table}
    \end{minipage}
    \par
    Since GF(2) satisfies the axioms required to be a field, we may consider
    vector spaces over GF(2), which may be endowed with a norm. In order to
    meaningfully define a norm on a vector space over GF(2), we define a
    function $|\cdot|:\mathrm{GF(2)}\rightarrow\mathbb{R}$
    which acts as an absolute value.
    \begin{equation}
        |0|=0,\quad|1|=1
    \end{equation}
    \par
    This definition of absolute value trivially satisfies
    non-negativity, positive-definiteness, multiplicativity, as well
    as the triangle inequality. So it is indeed sensible to define
    the absolute value for elements of GF(2) in this way.
    \begin{theorem*}
        There exists an infinite dimensional banach space S over GF(2)
        such that each bounded linear operator on S attains its norm.
    \end{theorem*}
    \begin{proof}
        Define an infinite dimensional banach space $S$ over GF(2) as follows:
        \begin{equation}
            S=\{\;(s_{1},s_{2},\dots )\;\;|\;\;s_{i}\neq0
                \text{ for finitely many }i\in\mathbb{N}\;\} 
        \end{equation}
        Vector addition and scalar multiplication are defined entry-wise.
        \begin{align}
            \mathbf{x}+\mathbf{y}=(x_{1}+y_{1},x_{2}+y_{2},\dots)
            &&\alpha\mathbf{x}=(\alpha{x}_{1},\alpha{x}_{2},\dots)
        \end{align}
        Note here that the operations $x_i+y_i$ and $\alpha x_i$ occur
        in GF(2). The space $S$ will be given the norm
        $\|\,\|_{S}:S\rightarrow\mathbb{R}$ defined by:
        \begin{align}
            \norm{\mathbf{x}}_{S}=
            \begin{cases}
                0,&\mathbf{x}=\mathbf{0}\\
                1,&\mathbf{x}\neq\mathbf{0}
            \end{cases}
        \end{align}
        Here, the zero vector is taken to be the sequence of all zeros.
        This space has the canonical basis, where $\mathbf{e}_{n}$ has a
        1 in the $n^{th}$ spot and 0 in the rest.
        \begin{equation}
            \mathbf{e}_n=(0,\dots,0,1,0,\dots )
        \end{equation}
        First we must verify that $S$ is a vector space. Since addition is performed
        entry-wise, associativity and commutativity are inherited properities of
        addition in the field. The identity element is the sequence of all zeros.
        Furthermore, since $1+1=0$ in GF(2), every element of $S$ is its own inverse
        with respect to addition. Now let $\alpha$ and $\beta$ be elements of GF(2),
        and $\mathbf{x}$ be in $S$. Then:
        \begin{equation}
            \alpha(\beta\mathbf{x})=
                \begin{cases}
                    \mathbf{0},&\alpha=0\textrm{ or }\beta=0\\
                    \mathbf{x},&\alpha=1\textrm{ and }\beta=1
                \end{cases}
        \end{equation}
        Similiary for $(\alpha\beta)\mathbf{x}$, and thus
        scalar multiplication is compatible with field multiplication. The identity
        element of scalar multiplication is $1\in \text{GF(2)}$. Finally, scalar
        multiplication trivially distributes over vector addition as well as field
        addition. Thus $S$ is a vector space.
        \par\hfill\par
        Now we verify that $S$ is a normed space. By the definition of $\| \,\|_S$, only
        the zero vector has norm zero, and all others have norm one. Thus
        positive-definiteness of the norm is satisfied. Now let $\alpha$ be
        an element of GF(2), and let $\mathbf{x}$ be a zero vector
        in $S$. If $\alpha$ is one then we observe:
        \begin{equation}
            \norm{\alpha\mathbf{x}}_{S}=\norm{\mathbf{x}}_{S}
                =1\norm{\mathbf{x}}_{S}
                =|\alpha|\norm{\mathbf{x}}_{S}
        \end{equation}
        If $\alpha$ is zero then instead we have:
        \begin{align*}
            \norm{\alpha\mathbf{x}}_{S}=\norm{\mathbf{0}}_{S}
                =0=0\norm{\mathbf{x}}_{S}
                =|\alpha|\norm{\mathbf{x}}_{S}
        \end{align*}
        In either case the result is that $\|\,\|_S$ is absolutely homogeneous.
        Moving forward, if $\mathbf{x}$ and $\mathbf{y}$ are distinct non-zero vectors
        in $S$, the norm of their sum will equal one. However the sum of their norms will be
        two, and thus be greater. For all other cases - one of them is the zero vector, both
        of them are the zero vector, or they are inverses - the triangle inequality trivially
        holds. Thus $\|\,\|_S$ is a norm on the vector space $S$.
        \par
        Lastly, in order for $S$ to be a Banach space, it must be complete with
        respect to the metric induced by the norm. If $\mathbf{x}$ and $\mathbf{y}$
        are non-zero vectors in $S$, then the distance between them is given as:
        \begin{align*}
            \norm{\mathbf{x}-\mathbf{y}}_{S}
                =\norm{\mathbf{x}+(-\mathbf{y})}_{S}
                =\norm{\mathbf{x}+\mathbf{y}}_{S}
                =\begin{cases} 
                    1,&\mathbf{x}\neq\mathbf{y}\\
                    0,&\mathbf{x}=\mathbf{y}
                \end{cases}
        \end{align*}
        We see that the metric induced by the norm is the discrete metric,
        so for a sequence in $S$ to be Cauchy, it must eventually be constant.
        Thus every Cauchy sequence in $S$ converges.
        Now that we have verified that $S$ is a Banach space, we must show
        that every operator $T$ in the space of bounded linear operators
        acting on $S$, $L(S)$ attains its norm. Let $T\in L(S)$ be a
        non-zero operator. The norm of $T$, $\|T\|$ is defined as:
        \begin{equation}
            \norm{T}=\sup\{\,\,\norm{T\mathbf{x}}_{S}\,\,\;|\,\,\;
                \mathbf{x}\in S,\,\,\;\norm{\mathbf{x}}_{S}=1\,\,\}
        \end{equation}
        For any vector in $S$, the norm of its image under $T$
        can only be either zero or one. Since $T$ was assumed
        not to be the zero operator, we obtain:
        \begin{equation}
            \norm{T}=1
        \end{equation}
        Furthermore, there must exist some $\tilde{\mathbf{x}}\in S$ such that
        $\norm{T\tilde{\mathbf{x}}}_{S}=1$, otherwise $T$ would have to be
        the zero operator. It should also be noted that the zero operator
        attains its norm via any point in $S$. Therefore, every operator
        in $L(S)$ attains its norm. This concludes the proof.
    \end{proof}
    
    \section*{Remarks}
    It is worth noting that by definition, $S$ cannot be a Hilbert space,
    since given some distinct and non-zero $\mathbf{x},\mathbf{y}\in S$,
    the parallelogram identity:
    \begin{equation}
        \norm{\mathbf{x}+\mathbf{y}}_{S}^{2}+\norm{\mathbf{x}-\mathbf{y}}_{S}^{2}
        =2\norm{\mathbf{x}}_{S}^{2}+2\norm{\mathbf{y}}_{S}^{2}
    \end{equation}
    Is true if and only if $\mathbf{x}=\mathbf{y}=\mathbf{0}$. Now consider the subset of sequences, $S_m^p$, which has $m$ 1's, all of which occur before or at the $p^{th}$ position in the sequence. There are $p$-choose-$m$ such sequences. 
    \begin{align}
        \bigcup\limits_{m=0}^p S_m^p
    \end{align}
    The union of these sets, as shown above, consists of all sequences with zeros after the $p^{th}$ entry. This union is a finite union of finite sets, and thus finite. The infinite union for all $p$ will be exactly our space $S$.
    \begin{align}
        S = \bigcup\limits_{p=0}^\infty \bigcup\limits_{m=0}^p S_m^p .
    \end{align}
    This is a countable union of finite sets, and thus $S$ is countable. Furthermore, because $S$ has the discrete topology, the only dense subset of $S$ is $S$ itself. Thus $S$ is a countable dense subset, and $S$ is trivially separable. 
    
\end{document}